\documentclass[12pt,twoside,leqno]{amsart}  

\usepackage{amsmath,amscd,amssymb,amsfonts}

\newcommand{\sL}{{\mathcal L}}
\newcommand{\sM}{{\mathcal M}}

\newcommand{\sO}{{\mathcal O}}

\newcommand{\C}{{\mathbb C}}

\renewcommand{\H}{{\mathbb H}}

\renewcommand{\L}{{\mathbb L}}

\renewcommand{\P}{{\mathbb P}}
\newcommand{\Q}{{\mathbb Q}}

\newcommand{\V}{{\mathbb V}}

\newcommand{\Z}{{\mathbb Z}}

\newcommand{\lra}{\longrightarrow }

\newcommand{\Ab}{{\theta}}

\newcommand{\End}{{\rm End}}

\theoremstyle{plain}
\newtheorem{thm}{Theorem}

\newtheorem{cor}[thm]{Corollary}
\newtheorem{prop}[thm]{Proposition}
\newtheorem{remark}[thm]{Remark}

\numberwithin{thm}{section}
\numberwithin{equation}{section}

\begin{document}
\title{Hodge classes associated to $1$-parameter families of Calabi-Yau $3$-folds}
\author{Pedro Luis del Angel}
\author{Stefan M\"uller-Stach}
\author{Duco van Straten}
\author{Kang Zuo}

\thanks{This work is entirely supported by DFG Sonderforschungsbereich/Transregio 45}                           
\begin{abstract}
We use $L^2$-Higgs cohomology to determine the Hodge numbers of the parabolic 
cohomology $H^1(\bar S, j_*\V)$, where the local system $\V$ arises from the
third primitive cohomology of family of Calabi-Yau threefolds over a curve $\bar S$.
The method gives a way to predict the presence of algebraic 2-cycles in the total space of the
family and is applied to some examples.
\end{abstract}

\maketitle

\section*{Introduction} 
We consider a smooth projective family of varieties $f: X \longrightarrow S$ over
a smooth curve $S$. After a finite pull-back we may assume that the family has 
a semi-stable completion $\bar f: \bar X \longrightarrow \bar S$, giving rise to a diagram
\[ \begin{array}{rcl}
X & \hookrightarrow & \bar X\\
f \downarrow&&\downarrow \bar f\\
S& \stackrel{j}{\hookrightarrow}& \bar S\\
\end{array}
\]
In such a situation the action of the local monodromy around each point of the boundary $D:=\bar S \setminus S$ on the local systems of $k$-cohomologies $R^kf_*\C_X$ is unipotent. These local systems carry a polarised variation of Hodge structures (VHS), where the Hodge filtration arises from the truncation of the
relative logarithmic deRham complex
\[ \Omega_{\bar X/\bar S}^{\bullet}(\log \Delta)\]
where $\Delta:=f^{-1}(D)$ is a reduced normal crossing divisor.\\

It is a very old idea that goes back to Poincar\'e to study cycles on $ \bar X$ 
by looking at their slicing by the fibres of the map. 
If $D \neq \emptyset$, then the Leray spectral sequence shows that the space of
 cohomology classes of $X$ that restrict to zero on fibres is the cohomology on $S$ of the local system 
of $k$-cohomologies
\[Ker(H^{k+1}(X) \longrightarrow H^0(S,R^{k+1}f_*\C_X))=H^1(S, R^kf_* \C_X).\]
It carries a canonical mixed Hodge structure and the subgroup 
\[H^1(\bar S, j_* R^kf_* \C_X)=H^1(\bar S, R^k\bar f_* \C_{\bar X}) \subset H^1(S, R^kf_* \C_X)\]
can be identified with the part of lowest weight $k+1$.
More generally, one can consider (irreducible) local systems $\V$ that
underly a weight $k$ VHS. The {\em intersection homology group} (also called the {\em parabolic cohomology group}) 
\[
H^1(\overline{S},j_*\V)
\]
then carries a natural Hodge structure of weight $k+1$, a result due to
S. Zucker, \cite{zucker}. In that paper, Zucker defined a local 
$L^2$-integrability condition for  $\mathcal C^\infty$-sections of  deRham complex twisted by the local system $\V$, using the Hodge metric and the 
Poincar\'e metric near the boundary points. Using the fundamental estimate 
of W. Schmid on the growth of the Hodge metric, he showed that a local section
 of it is $L^2$-integrable if and only if it lies in the monodromy weight 
filtration in the certain given range.
Using an $L^2$-Poincar\'e lemma, he showed that the subcomplex of $L^2$-integrable 
forms can be used to compute the intersection cohomology. In fact, using local information, one can define a
 certain complex $\Omega_{(2)}^{\bullet}(\V)$, whose hypercohomology computes
the intersection homology:
\[H_{(2)}^1(S,\V) := \H^1(\Omega_{(2)}^{\bullet}(\V))=H^1(\overline{S},j_*\V).\]
The complex carries a natural Hodge filtration and the associated
spectral sequence 
\[ \H^q(Gr_F^p \Omega_{(2)}^{\bullet}(\V)) \Rightarrow H^{p+q}_{(2)}(S,\V)\]
degenerates at $E_1$ and induces the pure Hodge structure on the limit.\\

In this paper we will focus attention to local systems that arise from 
non-trivial pencils of Calabi-Yau 3-folds. So we 
have $\bar S=\P^1$ and 
$f: X \to S=\P^1 \setminus D$ a smooth family of Calabi-Yau $3$-folds whose Kodaira-Spencer map is not identically zero, hence 
\[\dim_\C H^1(X_s,\Omega^2_{X_s})=\dim_\C H^{1}(X_s,\Theta_{X_s})=1.\]
As a consequence, the local system $\V:=R^3f_*\C_X$ of (primitive) middle 
cohomology is irreducible and has Hodge numbers $(1,1,1,1)$. The cohomology group
\[H:=H^1(\bar S, j_* \V)= H^1_{(2)}(S,\V)\]
carries a pure Hodge structure of weight $4$:
\[H=H^{4,0} \oplus H^{3,1} \oplus H^{2,2} \oplus H^{1,3}\oplus H^{0,4}.\]

In this paper, we will identify the Hodge numbers of $H$ in terms of geometric 
data of the family. We hope it will become clear to the reader how these ideas
can be used to obtain formulae in far more general situations.
Of particular interest is the Hodge space $H^{2,2}$, 
as any cycle $Z \in CH^2(\bar X)$ (homologous to zero on all fibres of $f$)
will give a class in it, as will be recalled in section 4. On the other 
hand, a class in $H^1(\bar S, j_*\V)$ also defines an extension of local 
systems, which is the same as the one
obtained from an inhomogenous Picard-Fuchs differential equations of 
type \[{\sL} \Phi_Z =g\] associated to the normal function $\Phi_Z$ 
of the cycle $Z$. Finally, we describe some examples that were recently 
considered in \cite{walcher1} and \cite{walcher2}.

\section{The $L^2$-Higgs complex}

Recall that a {\em logarithmic Higgs bundle} $(E,\theta)$ on $\bar S$ consists 
of a holomorphic vector bundle $E$, together with an $\sO_{\bar S}$-linear map
\[ E \stackrel{\theta}{\longrightarrow} E \otimes \Omega^1_{\bar S}(\log D).\]
In the geometrical case, $E$ is the sum of the Hodge bundles:
\[E=\oplus_{p+q=k} E^{p,q},\;\;\; E^{p,q}:=R^q\bar f_*(\Omega^{p}_{\bar X/\bar S}(\log \Delta))\]
and the Higgs field is obtained from cup-product with the Kodaira-Spencer class. 
By Griffiths transversality, it has components
\[ E^{i,j} \stackrel{\theta}{\longrightarrow} E^{i-1,j+1} \otimes \Omega^1_{\bar S}(\log D).\]

In a paper of Jost, Yang and the last named author \cite{jyz}, the local $L^2-$condition
of Zucker was adapted to the more general setting of logarithmic Higgs  bundles.

If $(E,\theta)$ is a logarithmic Higgs bundle on $\bar S$ and  $p \in D:=\bar S \setminus S$ is a boundary point, then the {\em residue}
of the Higgs-field $\theta$ at $p$ 
\[N=N_p:=Res_p(\theta) \in \End(E_p)\]
is a nilpotent endomorphism and so determines a weight filtration $W_\bullet$ on
the fibre $E_p$.\footnote{
In terms of the local system $\V$, the endomorphism $N_p$ can be identified with $\log T_p$, the logarithm of the local monodromy operator around the puncture $p$.}

(Recall that any nilpotent endomorphism $N$ of a vector space 
$V$ determines in a canonical way an increasing filtration on $V$, called the 
{\em monodromy weight filtration}:
\[
0 \subset W_{-m} \subset W_{-m+1} \subset \cdots \subset W_0 \subset W_{1} \subset \cdots \subset W_{m}=V. 
\]
It is defined as follows: if $N^{m+1}=0$ but $N^m \neq 0$ we put
\[
W_{m-1}={\rm Ker}(N^m), \quad W_{-m}={\rm Im}(N^m).
\]
Then inductively $W_k$ is constructed in such a way that $N(W_k)={\rm Im}(N) \cap W_{k-2} \subset W_{k-2}$ and 
$$
N^k: {\rm Gr}^W_{k}(V) \to {\rm Gr}^W_{-k}(V) 
$$
are isomorphisms.)\\ 
These weight filtrations on $E_p,\; p \in D$ can be used to define a
certain complex $(\Omega_{(2)}^{\bullet}(E),\theta)$, called the
{\em $L^2$-Higgs complex.} It is determined by sub-sheaves
\[ \Omega_ {(2)}^0(E) \subset E,\;\;\;\Omega_{(2)}^1(E) \subset E \otimes \Omega_{\bar S}^1(\log D)\]

which are defined near $p \in D$ by
\[
\Omega^0_{(2)}(E):=W_0 + tE, \quad \Omega^1_{(2)}(E):=(W_{-2} + tE) \otimes \Omega^1_{\bar S}(\log D).\]
Here $t$ is a local parameter near a boundary point $p=\{t=0\} \in D$ and
the weight filtration is defined by $N_p$.\\  

It can be shown \cite{jyz} that the hypercohomology of the $L^2-$ Higgs complex 
$(\Omega_{(2)}^{\bullet}(E),\theta)$ is isomorphic to the  
$L^2$-cohomology
\[ H^{k}_{(2)}(S,E)=\H^k(\Omega_{(2)}^{\bullet}(E),\theta).\]
\vskip 20pt
We now spell out the local structure of $(\Omega_{(2)}^{\bullet},\theta)$ for the case of logarithmic Higgs bundles of type $(1,1,1,1)$, so that
\[
E=E^{3,0} \oplus E^{2,1} \oplus E^{1,2} \oplus E^{0,3}
\]

and each summand is a line bundle.
At an ordinary point $p \in S$ the $L^2$-Higgs complex of course coincides with the Higgs bundle itself, so that we have
\[
\Omega^0_{(2)}=E^{3,0}\oplus E^{2,1}\oplus E^{1,2} \oplus E^{0,3}, \quad 
\Omega^1_{(2)}= \left(E^{3,0}\oplus E^{2,1}\oplus E^{1,2} \oplus E^{0,3} \right) \otimes \Omega^1_{\bar S}.
\]
For the points $p \in D$ one has to distinguish cases corresponding to the 
possible Jordan-forms of the endomorphism $N=N_p$.

There are two subcases that appear when $m=1$, that is $N^2=0$, but $N\neq 0$.\\

\subsection*{Point of type I } Here we have a single Jordan-block of the form
\[
N=\left( \begin{matrix} 0 & 0 & 0 & 0 \cr 0 & 0 & 1 & 0 \cr 0 & 0 & 0 & 0 \cr 0 & 0 & 0 & 0 \end{matrix} \right)
\]
The monodromy weight filtration has the following form: 
\[
0=W_{-1} \subset W_{0}={\rm Ker}(N)  \subset W_{1}=V.
\]
In this case in the sequence $E^{3,0} {\buildrel N \over \to} E^{2,1} {\buildrel N \over \to} E^{1,2}{\buildrel N \over \to}  E^{0,3}$ at the point $p=\{t=0\}$ the middle map must be an isomorphism, since $\dim  {\rm Ker}(N)=3$. Hence the left and right map must both be zero. 
Therefore the $L^2$-complex near $p$ looks like: 

\[
\begin{array}{cccccccccccl}
\Omega^0_{(2)}&=&0&\oplus &E^{3,0}&\oplus& tE^{2,1}&\oplus& E^{1,2}& \oplus &E^{0,3}&\\
\downarrow \theta&&\downarrow&&\downarrow&&\downarrow&&\downarrow&&\downarrow&\\
\Omega^1_{(2)}&=& \left(\; tE^{3,0}\right.&\oplus& tE^{2,1}&\oplus& tE^{1,2}& \oplus&  tE^{0,3}&\oplus&0&\hspace{-5mm}
\left. \right) \otimes \Omega^1_{\bar S}(\log D)\\
\end{array}
\]

\subsection*{Point of type II } Here we have two Jordan-blocks.
\[
N=\left( \begin{matrix} 0 & 1 & 0 & 0 \cr 0 & 0 & 0 & 0 \cr 0 & 0 & 0 & 1 \cr 0 & 0 & 0 & 0 \end{matrix} \right)
\]
The monodromy weight filtration now has the following form 
\[
0=W_{-1} \subset W_{0}={\rm Ker}(N)  \subset W_{1}=V.
\]
In this case in the sequence $E^{3,0} {\buildrel N \over \to} E^{2,1} {\buildrel N \over \to} E^{1,2}{\buildrel N \over \to}  E^{0,3}$
at the point $p=\{t=0\}$, the middle map must be zero, since $\dim  {\rm Ker}(N)=2$ and the left and right maps are dual to each other. 
Therefore the $L^2$-subcomplex has the local form:

\[
\begin{array}{cccccccccccl}
\Omega^0_{(2)}&=&0&\oplus &tE^{3,0}&\oplus& E^{2,1}&\oplus& tE^{1,2}& \oplus &E^{0,3}&\\
\downarrow \theta&&\downarrow&&\downarrow&&\downarrow&&\downarrow&&\downarrow&\\
\Omega^1_{(2)}&=& \left(\; tE^{3,0}\right.&\oplus& tE^{2,1}&\oplus& tE^{1,2}& \oplus&  tE^{0,3}&\oplus&0&\hspace{-5mm}
\left. \right) \otimes \Omega^1_{\bar S}(\log D).\\
\end{array}
\]

\subsection*{Case m=2} Assume that $N^2 \neq 0$ and $N^3=0$, so we may assume  
\[
N=\left( \begin{matrix} 0 & 1 & 0 & 0 \cr 0 & 0 & 1 & 0 \cr 0 & 0 & 0 & 0 \cr 0 & 0 & 0 & 0 \end{matrix} \right).
\]
At the point $p$ the first and the last maps in the sequence 
\[E^{3,0} {\buildrel N \over \to} E^{2,1} {\buildrel N \over \to} E^{1,2}{\buildrel N \over \to}  E^{0,3}\]
are dual to each other. If they are both isomorphisms, then $N^3=0$ implies that the middle map $E^{2,1} {\buildrel N \over \to} E^{1,2}$ is the zero map.
Hence we already have $N^2=0$. Therefore this case cannot happen.\\

\subsection*{Point of type III} Assume now that $m=3$, so $N^3 \neq 0$, so this 
case corresponds  to points of {\em maximal unipotent monodromy}, 
\[
N=\left( \begin{matrix} 0 & 1 & 0 & 0 \cr 0 & 0 & 1 & 0 \cr 0 & 0 & 0 & 1 \cr 0 & 0 & 0 & 0 \end{matrix} \right).
\]
The monodromy weight filtration has the following form:
\[
0 \subset W_{-3}=W_{-2}={\rm Ker}(N) \subset W_{-1}=W_{0}={\rm Ker}(N^2)  \subset W_{1}=W_{2}={\rm Ker}(N^3) \subset V.
\]
Since $N^3 \neq 0$ and the $E^{p,q}$ have rank one, the composition 
\[E^{3,0} {\buildrel N \over \to} E^{2,1} {\buildrel N \over \to} E^{1,2}{\buildrel N \over \to}  E^{0,3}\]
at the point $p=\{t=0\}$ is an isomorphism. The $L^2$-subcomplex takes the following form 

\[
\begin{array}{cccccccccccl}
\Omega^0_{(2)}&=&0&\oplus &tE^{3,0}&\oplus& tE^{2,1}&\oplus& E^{1,2}& \oplus &E^{0,3}&\\
\downarrow \theta&&\downarrow&&\downarrow&&\downarrow&&\downarrow&&\downarrow&\\
\Omega^1_{(2)}&=& \left(\; tE^{3,0}\right.&\oplus& tE^{2,1}&\oplus& tE^{1,2}& \oplus&  E^{0,3}&\oplus&0&\hspace{-5mm}\left. \right)\otimes \Omega^1_{\bar S}(\log D).\\
\end{array}
\]

\section{Hodge Filtration}

One of the main points of the constructions of the $L^2$-Higgs complex is that one 
obtains a more or less {\em  explicit} description of the Hodge-decomposition on 
$H^1(\overline{S},j_*\V)$.
The obvious Hodge filtration 
\[ F^k \subset F^{k-1} \subset \ldots \subset F^0,\;\;\;F^p := \bigoplus_{i\ge p} E^{i,k-i}\]
on $E$ defines by intersection  with the $L^2$-subcomplex $\Omega_{(2)}^\bullet(E)$ a filtration
and the corresponding hypercohomology spectral sequence degenerates at $E_1$.

It appears from description of the previous section, that the complex $\Omega_{(2)}^\bullet(E)$
is in fact the direct sum of graded pieces
\[ \Omega_{(2)}^\bullet(E)=\oplus_{p+q=k}\Omega_ {(2)}^{\bullet}(E)^{p,q} \]
where
\[\Omega_{(2)}^0(E)^{p,q}:=\Omega_{(2)}^0(E) \cap E^{p,q}\]
\[\Omega_{(2)}^1(E)^{p,q}:=\Omega_{(2)}^1(E) \cap E^{p-1,q+1} \otimes \Omega_{\bar S}^1(\log D)\]

In the case of a logarithmic Higgs bundle of type $(1,1,1,1)$ considered before,
the $F$-graded pieces have the following form:

\[
\begin{array}{|c|ccc|}
\hline
&0&&1\\[1mm]
\hline
&&&\\
Gr_F^4\Omega_{(2)}^\bullet(E)&0&\lra&\Omega^{1}_{(2)}(E)^{3,0}\\[4mm]
\hline
&&&\\
Gr_F^3\Omega_{(2)}^\bullet(E)&\Omega_{(2)}^0(E)^{3,0}&\stackrel{\theta}{\lra}&\Omega^{1}_{(2)}(E)^{2,1}\\[4mm]
\hline
&&&\\
Gr_F^2\Omega_{(2)}^\bullet(E)&\Omega_{(2)}^{0}(E)^ {2,1}&\stackrel{\theta}{\lra}&\Omega^{1}_{(2)}(E)^{1,2}\\[4mm]
\hline
&&&\\
Gr_F^1\Omega_{(2)}^\bullet(E)&\Omega_{(2)}^{0}(E)^{1,2}&\stackrel{\theta}{\lra}&\Omega^1_{(2)}(E)^{0,3}\\[4mm]
\hline
&&&\\
Gr_F^0\Omega_{(2)}^\bullet(E)&\Omega_{(2)}^{0}(E)^{0,3}&\stackrel{\theta}{\lra}&0\\[4mm]
\hline
\end{array}
\]

From this we obtain the following expression for the Hodge spaces $H^{p,q}$ of $H^1_{(2)}(S,E)$:

\begin{thm}
\[
\begin{array}{rcl} 
H^{4,0} & =&H^0(\bar S,E^{30}\otimes \Omega_{\bar S}^1)  \\[2mm]
H^{3,1} & =&H^0(\bar S,\Omega^1_{(2)}(E)^{2,1}/\theta(\Omega^0_{(2)}(E)^{3,0})) \\[2mm]
H^{2,2} & =&H^0(\bar S,\Omega^1_{(2)}(E)^{1,2}/\theta(\Omega^0_{(2)}(E)^{2,1})) \\[2mm]
H^{1,3} & =&H^0(\bar S,\Omega^1_{(2)}(E)^{3,0}/\theta(\Omega^0_{(2)}(E)^{1,2})) \\[2mm]
H^{0,4} & =&H^1(\bar S, E^{03})\\[2mm]
\end{array}
\]
\end{thm}
\begin{proof}
From the local description of the $L^2$- Higgs complex one sees 
\[\Omega_{(2)}^{1}(E)^{3,0}=E^{3,0}\otimes \Omega_{\bar S}^1,\;\;\;\Omega_{(2)}^0(E)^{0,3}=E^{0,3}\]
which gives the $H^{4,0}$ and $H^{0,4}$-spaces.
For the other three we remark that the map $\theta$ is injective, so the
two-term complex is quasi-isomorphic to the corresponding cokernel (at spot one).
\end{proof}

\section{Numerical consequences}

We can use the above description of the $L^2$-Higgs complex and its 
Hodge-graded pieces to obtain explicit formulae for the Hodge numbers of
\[ H:=H^1_{(2)}(S,E)\]
in terms of elementary geometric data. As before, $E$ is a logarithmic
Higgs-bundle on a curve $\bar S$ (of genus $g(\bar S)$) of type $(1,1,1,1)$.

Using the local descriptions of the previous section, we can give a global description of
the $L^2$-Higgs complex.

\begin{prop} Let 
\[\sO_{\bar S}(-I),\;\;\;\;\sO_{\bar S}(-II),\;\;\;\;\sO_{\bar S}(-III),\;\;\;\;\]
denote the ideal sheaves of the points of type $I$, $II$ and $III$ respectively.
Then the graded pieces of the terms of the $L^2$-Higgs complex $\Omega^\bullet_{(2)}(E)$
are given by

\[
\begin{array}{|ccl|ccl|}
\hline
&&&&&\\
&&&\Omega^{1}_{(2)}(E)^{3,0}&=&E^{3,0} \otimes \Omega^1_{\bar S}\\[4mm]
\hline
&&&&&\\
\Omega_{(2)}^0(E)^{3,0}&=&E^{3,0}(-II-III) &\Omega^{1}_{(2)}(E)^{2,1}&=&E^{2,1} \otimes \Omega_{\bar S}^1 \\[4mm]
\hline
&&&&&\\
\Omega_{(2)}^{0}(E)^ {2,1}&=&E^{2,1}(-I-III)  &\Omega^{1}_{(2)}(E)^{1,2}&=&E^{1,2} \otimes \Omega_{\bar S}^1\\[4mm]
\hline
&&&&&\\
\Omega_{(2)}^{0}(E)^{1,2}&=&E^{1,2}(-II) &\Omega^1_{(2)}(E)^{0,3}&=&E^{0,3}(III) \otimes \Omega_{\bar S}^1   \\[4mm]
\hline
&&&&&\\
\Omega_{(2)}^{0}(E)^{0,3}&=&E^{0,3}&&&\\[4mm]
\hline
\end{array}
\]
\end{prop}
\begin{proof} This can be read off immediately by looking at the positions of the $t$'s in the local
description of the $L^2$-Higgs complex.
\end{proof}

We put 
\[a:=\deg(E^{3,0}),\;\;\;b:=\deg(E^{2,1}),\]
so that by self-duality of the Higgs-bundle we have
\[  \deg(E^{0,3})=-a,\;\;\;\deg(E^{1,2})=-b\]

Let $A, B, C$ be the numbers of singular points of type $I, II, III$ described above.

\begin{thm} The Hodge numbers of $H:=H^1_{(2)}(S,\V)$ are:
\begin{align*} 
h^{4,0} & = h^{0,4}=a-1+g(\bar S), \\
h^{3,1} & = h^{1,3}=B+C+b-a-2+2g(\bar S), \\
h^{2,2} & = A+C-2b-2+2g(\bar S).
\end{align*}
In particular we have 
\[
\dim_\C H=A+2B+3C+8(g(\bar S)-1). 
\]
\end{thm}
\begin{proof} We use the description of the Hodge spaces from the previous section.
We have $H^{4,0}=H^0(\bar S,E^{3,0} \otimes \Omega_{\bar S}^1)$ and  so by Riemann-Roch
$h^{4,0}=a+2g-2-g+1=a-1+g$
as $H^1(\bar S,E^{3,0} \otimes \Omega_{\bar S}^1)=0$ by positivity of $E^{3,0}$ ($g=g(\bar S)$).

Recall that the space $H^{3,1}$ is represented as global sections of the cokernel of the map 
$\Omega_{(2)}^0(E)^{3,0}\stackrel{\theta}{\lra} \Omega_{(2)}^1(E)^{2,1}$, so we have to count the zeros 
of a bundle map
\[E^{3,0}(-II-III) \lra E^{2,1} \otimes \Omega_{\bar S}^1.\]
Their number is given by the difference in degree of the bundles involved, so
\[h^{3,1}=\deg E^{2,1} \otimes \Omega_{\bar S}^1-\deg E^{3,0}(-II-III)=b+2g-2-a+B+C\]

Similarly, the space $H^{2,2}$ is represented as global sections of the cokernel of the map 
$\Omega_{(2)}^{0}(E)^{2,1} \stackrel{\theta}{\lra} \Omega_{(2)}^1(E)^{1,2}$, so we have to 
count the zeros of a bundle map
\[ E^{2,1}(-I-III) \lra E^{1,2} \otimes \Omega_{\bar S}^1.         \]
and one obtains the dimension in a similar way.
By duality, the other Hodge numbers are determined. 
\end{proof}

\begin{cor}
In the special case $\dim_\C H \le 1$ we have $H=H^{2,2}$ and it follows that
\[g(\bar S)=0,\;\;\;a=1,\;\;\;b=\frac{1}{2} (A+C-2-\dim_\C H)\]. 
\end{cor}

\begin{remark}
The $3$-fold iterated Kodaira-Spencer map $\theta$ gives rise to a homomorphism 
\[
\theta^{(3)}: (E^{3,0})^{\otimes 2} \longrightarrow S^3 \Omega^1_{\bar S}(\log D) 
\]
which is called the Griffiths-Yukawa coupling.  
The number of zeroes of $\theta^{(3)}$ counted with multiplicity is therefore given by 
the formula
\[
\sharp \{p \in \P^1 \mid \theta^{(3)}(p)=0\}=3(N-2+2g(\bar S))-2a. 
\]
Here
\[N:=A+B+C\]
denotes the total number of singularities.
\end{remark}

\begin{remark}
For a non-trivial, irreducible local system $\V$ there is
a simple classical formula for the dimension
\[h^1(\V):=\dim_\C H^1(\P^1,j_* \V)\]
in terms of the {\em ramification indices}
\[
R(p):={\rm Rank}(\V)- {\rm dim}(\V^{I(p)}). 
\]
where $I(p)$ is the local monodromy around the point $p$:
\begin{prop}
\[
h^1(\V)= \sum_{p \in D} R(p) +(2g(\bar S)-2) {\rm Rank}(\V). 
\]
\end{prop}
\begin{proof}
From topology one obtains
\[\chi(j_*\V)=\chi(S) \cdot {\rm Rank}(\V)+\sum_{p \in D}\dim \V^{I(p)}=(2-2g) {\rm Rank}(\V)-\sum _{p \in D} R(p).\]
From the exact sequence
\[ 0 \lra j_!\V \lra j_*\V \lra i^*j_*\V \lra 0\]
(where $i:D \hookrightarrow \bar S$) one gets $H^2(j_!\V)=H^2(j_*\V)$ and
from Poincar\'e duality one gets
\[H^2(j_!\V)=H^2_c(S,\V)=H^0(S,\V^*)^*\]
As $H^0(S,\V^*)=0$ the result follows.
\end{proof} 

\end{remark}

\section{Cycle classes}

Recall that the space $H^1(S,\V)$ is naturally identified with the space of 
{\em extension classes}
\[ 0 \lra \V \lra \widehat{\V} \lra \C_S \lra 0\]
of local systems on $S$. The parabolic subspace 
$H^1(\bar S,j_*\V) \subset H^1(S,\V)$ consists precisely of those extensions which 
split locally near each point $p \in D$, see e.g. \cite{stiller2}. If the local 
system is described as horizontal sections or solutions to a homogeneous differential equation on $S$
\[ \sL \Phi=0,\]
then extensions are described in terms of inhomogeneous equations of the form
\[ \sL \Psi =g\]
where $g \in \sM_S$ is meromorphic function on $S$.
The local system of solutions to the extended equation
\[
\widehat{\sL} \Phi=0, \quad \widehat {\sL}=(g\frac{d}{dt}-g'){\sL}.
\]
make up the local system $\widehat{\V}$.
This leads to a description of the parabolic cohomology by an exact sequence of
the form
\[ 0 \lra \sL (\sM_S) \lra \sM^{para}_S \lra H^1(\bar S,j_*\V ) \lra 0,\] 
where $\sM^{para}$ is the set of meromorphic functions on $\bar S$ that satisfy a
{\em parabolic condition} in the points $p \in D$. We refer to \cite{stiller1} and
\cite{stiller2} for details of this construction.\\ 

Consider as in the introduction a semi-stable morphism $\bar f : \bar X \to \bar S$ which extends a 
smooth map $f:X \lra S$. We set $\Delta= \bar f^{-1}(D)$ be the inverse image of  $D:=\bar S \setminus S$ , a reduced simple normal crossing divisor. We consider algebraic cycles  $Z \in CH^2(\bar X)$, whose 
class restrict to zero on the fibres of $f$, i.e. are in the kernel of the restriction map 
$H^4(\bar X) \lra H^4(fibre)$.
Such a cycle determines a locally split extension
\[0 \lra \V \lra \widehat{\V} \lra \C_S \lra 0\]
of local systems on $S$ that we will describe now. Its fibre over $t \in S$ is
\[  0 \lra H^3(X_t) \lra \widehat{\V}_t \lra \C \lra 0\]
\[ \widehat{\V}_t=\rho^{-1}(\C\cdot[Z_t]),\;\;\rho: H^3(X_t\setminus |Z_t|) \lra H^4_{|Z_t|}(X_t)  \]
where $X_t:=f^{-1}(t)$ and $Z_t:=Z \cap X_t$.  
The choice of a $3$-form $\omega \in H^0(S,E^{3,0})$ determines the Picard-Fuchs differential operator 
$\sL$ which describes the local system $\V=R^3f_*\C_X$ and has all periods
\[\Phi:=\int_{\gamma_t} \omega,\;\;\gamma_t \in H_3(X_t).\]
as solutions.
The cycle $Z$ and the form $\omega$ determine a {\em normal function} in the sense of 
Poincar\'e-Griffiths:
\[\Phi_Z(t)=\int_{\Gamma_t} \omega\]
where the 3-chain $\Gamma_t$ has the property that $\partial \Gamma_t= Z_t$. 
It satisfies an inhomogenous Picard-Fuchs ODE of the form
\[
{\sL} \Phi_Z=g,
\]
where the source function $g$ is a meromorphic function on $\bar S$ \cite{AM}.\\

There is a Higgs part to this story: each cycle $Z \in CH^2(\bar X)$ will also determine
an element in the group $H:=H^1_{(2)}(S,E)$, which in fact will lie in $(2,2)$ component
$H^{2,2}$ of $H$. Let us  indicate how to obtain directly a class in  
\[
H^0(\bar S,E^{1,2} \otimes \Omega^1_{\bar S}(\log D)/\theta(E^{2,1}))
\]
or rather its $L^2$-version
\[
H^0(\bar S,\Omega_{(2)}^1(E)^{1,2}/\theta(\Omega_{(2)}^{0}(E)^{2,1})).
\]

From the class of $Z$ in $H^4(\bar X)$ we first obtain by restriction to $X$ a class in $H^2(\bar X, \Omega^2_{\bar X}(\log \Delta))$ which is mapped to a class in 
\[H^0(\bar S, R^2 \bar f_* \Omega^2_{\bar X}(\log \Delta))\]
by the edge-homomorphism 
$
H^2(\bar X, \Omega^2_{\bar X}(\log \Delta)) \to H^0(\bar S, R^2 \overline{f}_* \Omega^2_{\bar X}(\log \Delta))$
of the Leray spectral sequence. The exact sequence
\[
0 \to \overline{f}^* \Omega^1_{\bar S}(\log D) \otimes \Omega^1_{\bar X/\bar S}(\log \Delta) \to 
\Omega^2_{\bar X}(\log \Delta) \to \Omega^2_{\bar X/\bar S}(\log \Delta) \to 0
\]
gives by taking the direct image under $\bar f_*$ a long exact sequence of sheaves on $\bar S$:
\[
\ldots \lra E^{2,1} \stackrel{\theta}{\lra} \Omega_{\bar S}^1(\log D)\otimes E^{1,2}\lra R^2\bar f_*\Omega^2_{\bar X}(\log \Delta) \lra R^2\bar f_*\Omega^2_{\bar X/\bar S}(\log \Delta) \lra \ldots
\]
From this sequence we get 
\[H^0(E^{1,2} \otimes \Omega^1_{\bar S}(\log D)/E^{2,1}) =Ker( H^0(\bar S, R^2 \bar f_* \Omega^2_{\bar X}(\log \Delta)) \to H^0(\bar S, R^2 \bar f_* \Omega^2_{\bar X / \bar S}(\log \Delta))).
\]
As $Z$ is assumed to be homologous to zero on the fibres, the class of $Z$
indeed is in the kernel and hence we obtain an element
\[  \delta(Z) \in  H^0(E^{1,2} \otimes \Omega^1_{\bar S}(\log D)/\theta(E^{2,1}))   \]
which is called {\em Griffiths' infinitesimal invariant} of $Z$.
However, we know a priori that the class of $Z$ is of type $(2,2)$, so that
the element lies in the  $H^{2,2}$-subspace of $H_{(2)}^1(S,R^3f_*\C_{X})$. Hence
we must have that in fact
\[
\delta(Z) \in H^0(\bar S,\Omega^{1}_{(2)}(E)^{1,2}/\theta(\Omega_{(2)}^{0}(E)^{2,1})).
\]

One can also understand $H_{(2)}^1(E,\theta)$ in terms of {\em extensions of Higgs-bundles}.
If $(E,\theta)$ is the Higgs bundle associated to $\V$, then there is a Higgs bundle 
$(\widehat{E},\widehat{\theta})$ associated to $\widehat \V$. It is obtained as sum of
graded pieces for the Hodge filtration of the Deligne extension of $\widehat \V$. It sits
in an exact sequence 
\[
0 \to (E,\theta) \to (\widehat E,\widehat \theta) \to (\sO_{\bar S},0) \to 0. 
\]
where $(\sO_{\bar S},0)$ is the trivial Higgs bundle.
This extension produces an extension of $L^2$-Higgs complexes
\[
0 \lra \Omega_{(2)}^\bullet (E) \lra \Omega_{(2)}^\bullet (\widehat{E}) \lra \Omega_{(2)}^\bullet (\sO) \lra 0
\]
and the connecting homomorphism of hypercohomology gives an exact sequence
\[
 \ldots \lra H^0(\bar S,\sO) \stackrel{\partial}{\lra} H^1_{(2)}(S,E) \lra H^1_{(2)}(S,\widehat{E})) \lra \ldots
\]
Clearly one has
\[ \partial(1) =\delta(Z).\]

\section{Examples}

\subsection{Mirror Quintic}
The Dwork pencil is the following pencil of quintics in $\P^4$ 
\[ x_1^5+x_2^5+x_3^5+x_4^5+x_5^5-5\psi x_1x_2x_3x_4x_5=0\]
which plays a prominent role in mirror symmetry \cite{candelas}.

Apart from the degeneration into the union of $5$ hyperplanes $x_1x_2x_3x_4x_5=0$ occurring for $\psi=\infty$, the fibres aquire singularities over the fifth 
roots of unity $\psi^5=1$. In fact, these fibres have $125$ nodes, which form a
single orbit under a group $G \approx (\Z/5)^3$ that acts on the whole pencil 
by scalings of the coordinates that preserve each monomial of the equation.

The quotient by the pencil by $G$ has a birational model as pencil of 
quintics \cite{meyer}:
\[(x_1+x_2+x_3+x_4+x_5)^5-(5\psi)^5x_1x_2x_3x_4x_5=0\] 
A crepant resolution of the general fibre is a Calabi-Yau 3-fold with 
Hodge numbers $h^{1,2}=1$, $h^{1,1}=101$, called the {\em mirror quintic}.

The local system $\V=R^3f_*\C$ of this family is unipotent and has $6$ singular
fibres; there is one point of type III over $\psi=\infty$ and the $5$ points of type II over
 the fifth unit roots. So we have $N=6$, $A=5$, $B=0$, $C=1$ and from the dimension formula 
one finds $h^1(\V)=0$. From our formulas we
deduce $a=1$ and $b=2$.
It follows that Griffiths-Yukawa Coupling $\theta^{(3)}$ has $10$ zeroes. It
vanishes twice at the fifth roots of unity.

\medskip After taking a further base change $z \mapsto z^2$, we obtain a family with $A=10$, $B=0$, $C=1$ and $N=11$. By the dimension formula we have $h^1(\V)=5$. As we pull back a unipotent system, the Hodge bundles are obtained by
 base-change. So we have $a=2, b=4$. Our formula gives $h^{4,0}=h^{3,1}=h^{2,2}=1$.\\

The Yukawa-coupling $\theta^{(3)}$ vanishes twice at each of the $10$ 
points of type I. Furthermore, $\theta^{(3)}$ now has a triple zero at 
$0$ due to the ramification, whereas at $\infty$ the ramification does not 
produce a zero, as the pull-back of a point of type III remains of type III. 
So $\theta^{(3)}$ has $23$ zeros in total, in accordance with our formula.\\ 

Hence, $H^1(\bar S, j_*\V)$ is a remarkable Hodge structure of dimension $5$, all
whose Hodge numbers are one. In particular there should be one interesting
$2$-cycle on the total space. It was determined explicitly by Walcher, using
matrix factorizations, \cite{walcher1}.

\subsection{Hypergeometric pencils}

Note that the pencil of the mirror quintic of the previous paragraph 
is the pull-back under the map $z=(5\psi)^{-5}$ of the pencil
\[z(x_1+x_2+x_3+x_4+x_5)^5-x_1x_2x_3x_4x_5=0\]   
The local system $\L=R^3f_*\C$ of this family can be identified with the
local system of solutions to the hypergeometric operator
\[\Ab^4-5z(5\Ab+1)(5\Ab+2)(5\Ab+3)(5\Ab+4),\;\;\;\;(\Ab=z\partial/\partial z)\]
which has three singular points $z=0,z=5^{-5}, \infty$.
The monodromy around $z=0$ of of type $III$, around $z=5^{-5}$ of type $I$ and
around $z=\infty$ it is of order five. The dimension formula gives
\[\dim_\C H^1(\bar S,j_*\L)=1+3+4-8=0\]
When we pull-back over the map $z \lra z^2$, we get a local system $\widetilde{\L}$
with
\[\dim_\C H^1(\bar S,j_*\widetilde{\L})=2+3+4-8=1\]
so there is a cycle. We will need an extension our theory to the quasi-unipotent
case in order to see via $L^2$-Higgs theory that this is indeed a $(2,2)$-class.
This gives an explanation for the appearance of $\sqrt{z}$ in the inhomogenous term
\[\sL \Phi = \frac{c}{(2\pi i)^2}\sqrt{z},\;\; c\in \Q\]
for the inhomogeneous Picard-Fuchs equation of the cycle in $H^{2,2}$.\\

In total there are $14$ such hypergeometric local systems that carry a variation
 of Hodge structures \cite{doranmorgan}: they all arise as mirror families of 
Calabi-Yau complete intersections in weighted projective spaces. 
We consider pull-backs $z \mapsto z^e$ of order $e$, which have 1 point of
type III, $e$ points of type $I$ and possibly a further singularity at $\infty$.   

\newpage
{\bf Seven cases with semi-simple monodromy at infinity.}

\[
\begin{array}{|l|l||c|c|c|c|c|}
\hline
\textup{Operator}&\textup{Model}&e&h^1&h^{4,0}&h^{3,1}&h^{2,2}\\
\hline
\Ab^4-5z(5\Ab+1)(5\Ab+2)(5\Ab+3)(5\Ab+4)&\P^4[5]&1&0&0&0&0\\
                                        &       &2&1&0&0&1\\
                                        &       &5&0&0&0&0\\
                                        &       &10&1&1&1&1\\
\hline
\Ab^4-36z(6\Ab+1)(3\Ab+1)(3\Ab+2)(6\Ab+5)&\P(1,1,1,1,2)[6]&1&0&0&0&0\\
                                                  &       &2&1&0&0&1\\
                                                  &       &6&1&0&0&1\\
\hline
\Ab^4-16z(8\Ab+1)(8\Ab+3)(8\Ab+5)(8\Ab+7)&\P(1,1,1,1,4)[8]&1&0&0&0&0\\
                                                  &       &2&1&0&0&1\\
                                                  &       &8&3&&&\\
\hline
\Ab^4-80z(10\Ab+1)(10\Ab+3)(10\Ab+7)(10\Ab+9)&\P(1,1,1,2,5)[10]&1&0&0&0&0\\
                                                       &       &2&1&0&0&1\\
                                                       &       &10&5&&&\\
\hline
\Ab^4-48z(6\Ab+1)(4\Ab+1)(4\Ab+3)(6\Ab+5)&\P(1,1,1,1,1,1,2)[3,4]&1&0&0&0&0\\
                                                       &       &2&1&0&0&1\\
                                                       &       &4&1&0&0&1\\
                                                       &       &6&3&&&\\
                                                       &       &12&7&&&\\
\hline
\Ab^4-144z(12\Ab+1)(12\Ab+5)(12\Ab+7)(12\Ab+11)&\P(1,1,1,1,2,12)[2,12]&1&0&0&0&0\\
                                                       &       &2&1&0&0&1\\
                                                       &       &12&7&&&\\
                                           
\hline
\Ab^4-12z(4\Ab+1)(3\Ab+1)(3\Ab+2)(4\Ab+3)&\P(1,1,1,2,2,3)[4,6]&1&0&0&0&0\\
                                                       &       &2&1&0&0&1\\
                                                       &       &3&0&0&0&0\\
                                                       &       &4&1&0&0&1\\
                                                       &       &12&7&&&\\
\hline 
\end{array}
\]

{\bf Three cases with one Jordan-block at $\infty$:}
\[
\begin{array}{|l|l||c|c|c|c|c|}
\hline
\textup{Operator}&\textup{Model}&e&h^1&h^{4,0}&h^{3,1}&h^{2,2}\\
\hline
\Ab^4-16z(4\Ab+1)(2\Ab+1)^2(4\Ab+3)&\P^5[2,4]&1&0&0&0&0\\
                                       &     &2&0&0&0&0\\
                                       &     &4&0&0&0&0\\
                                       &     &8&4&1&1&0\\
\hline
\Ab^4-12z(3\Ab+1)(2\Ab+1)^2(3\Ab+2)&\P^6[2,2,3]&1&0&0&0&0\\
                                       &      &2&0&0&0&0\\
                                       &      &3&0&0&0&0\\
                                       &      &6&2&&&\\

\hline
\Ab^4-48z(6\Ab+1)(2\Ab+1)^2(6\Ab+5)&\P^5[2,6]&1&0&0&0&0\\
                                       &      &2&0&0&0&0\\
                                       &      &6&2&&&\\
\hline
\end{array}
\]

For example, the pull-back of the $\P^5[2,4]$-operator by
$z \mapsto z^4$ has $A=5$, $C=1$, so still $h^1=3+5-8=0$, and thus 
it follows $a=1, b=2$. Making an additional quadratic pull-back 
$z \mapsto z^2$ gives a local system with $a=2, b=4$, $A=9$,$B=0$, 
$C=1$, so we get $h^{4,0}=h^{3,1}=1, h^{2,2}=0$!\\
 
{\bf Three cases with two Jordan-blocks at infinity:}
\[
\begin{array}{|l|l||c|c|c|c|c|}
\hline
\textup{Operator}&\textup{Model}&e&h^1&h^{4,0}&h^{3,1}&h^{2,2}\\
\hline
\Ab^4-9z(3\Ab+1)^2(3\Ab+2)^2&\P^5[3,3]&1&0&0&0&0\\
                            &          &2&1&0&0&1\\
                            &          &3&0&0&0&0\\
                            &          &6&3&1&0&1\\
\hline
\Ab^4-16z(4\Ab+1)^2(4\Ab+3)^2&\P^5(1,1,1,1,2,2)[4,4]&1&0&0&0&0\\
                                       &            &2&1&0&0&1\\
                                       &            &4&1&0&0&1\\
                                       &            &8&5&1&0&3\\

\hline
\Ab^4-144z(6\Ab+1)^2(6\Ab+5)^2&\P^5[2,6]&1&0&0&0&0\\
                               &        &2&1&0&0&1\\
                               &        &6&3&&&\\
\hline
\end{array}
\]

For example, the pull-back of the $\P^5[3,3]$-operator by
$z \mapsto z^3$ has $A=3$, $B=1$, $C=1$, so still $h^1=3+2+3-8=0$, and thus 
it follows $a=1, b=1$. Making an additional quadratic pull-back 
$z \mapsto z^2$ gives a local system with $a=2, b=2$, $A=6$,$B=1$, 
$C=1$, so we get $h^{4,0}=1,h^{3,1}=0, h^{2,2}=1$.\\

{\bf One case with a maximal Jordan block at infinity:}
\[
\begin{array}{|l|l||c|c|c|c|c|}
\hline
\textup{Operator}&\textup{Model}&e&h^1&h^{4,0}&h^{3,1}&h^{2,2}\\
\hline
\Ab^4-16z(2\Ab+1)^4&\P^7[2,2,2,2]&1&0&0&0&0\\
                            &    &2&0&0&0&0\\
                            &    &2k &2k-2 &k-1&0&0\\
\hline
\end{array}
\]
Pulling-back by $z \mapsto z^2$ produces a local system with $A=2,B=0,C=2$,
so $h^1=0$, hence $a=1,b=1$. A further pull-back by $z \mapsto z^k$, gives
$A=2k,B=0,C=2$, $a=k, b=k$ and hence $h^{4,0}=k-1$,$h^{3,1}=h^{2,2}=0$; we 
never get any $(2,2)$-classes.


This gives some apriori explanation for which cases one expects the 
appearance of $\sqrt{z}$ in the inhomogenous term
\[\sL \Phi = \frac{c}{(2\pi i)^2}\sqrt{z},\;\; c\in \Q\]
for Picard-Fuchs equation of the cycle in $H^{2,2}$, \cite{walcher1},\cite{walcher2}.

\subsection{A Grassmannian Example}
In \cite{AESZ} over $300$ examples of differential equations of
Calabi-Yau type are collected. Conjecturally all of these have a
geometrical origin and underlie a variation of Hodge structures
of type $(1,1,1,1)$. After the $14$ hypergeometric examples there
are cases that arise from mirror symmetry for complete intersections
in a Grassmannian.
 
The canonical class of the six dimensional Grassmannian $Z:=Gr(2,5)$ is $5 H$, where 
$H$ is the ample generator of $H^2(Z)$ which defines the Pl\"ucker-embedding of $Z$ in 
$\P^{10}$. Consequently, the complete intersection of $Z$ with three hypersurfaces of degree $1,2,2$ is a smooth Calabi-Yau threefold with $h^{1,1}=1$.
In \cite{bcks} a mirror family was described. It is a pencil with $4$ singular points 
$0, \alpha, \beta, \infty$, where
\[\alpha:=\frac{1}{32}(5\sqrt{5}-11),\;\;\;\beta:=\frac{1}{32}(-5\sqrt{5}-11)\]

The Picard-Fuchs operator for this family was computed in \cite{bcks}:
\[\Ab^4-4z (11 \Ab^2+11\Ab+3)(2\Ab+1)^2-16x^2(2\Ab+1)^2(2\Ab+3)^2,\;\;\;(\Ab=z\partial/\partial z).\]
which has a holomorphic solution
\[\Phi(z)=\sum_{n=0}^\infty A_n z^n,\;\;\;A_n:={ 2n \choose n}^2\sum_{k=0}^n {n \choose k}^2{n+k \choose k}\]
(It is Nr.25 in the list \cite{AESZ}).
Around $0$ we have maximal unipotent monodromy, around the points $\alpha$ and
$\beta$ we have a single Jordan-block of size two. 
After a base change $x \mapsto x^2$ we are in a unipotent situation also at $\infty$. The points $\alpha$ and $\beta$ each have 
two preimages of type $A$. The points $\infty$ is of type $B$ (two Jordan blocks). This implies immediately that $1=h^1(\V)=h^{2,2}(\V)$ and the other Hodge numbers are zero. The constants are $A=4$, $B=1$, $C=1$, $a=1$ and $b=1$.

\subsection{A remarkable Hadamard Product } 
The second order operator
\[ \Ab^2 -z(32\Ab^2+ 32\Ab+12)-256(\Ab+1)^2\]
is Picard-Fuchs operator of a rational elliptic surface and defines a local system
$\L$ of rank two. It has a unique period
that is holomorphic near $0$ which has an expansion $\phi(z)=1+12z+164z^2+\ldots+$
\[\phi(z)=\sum_{n=0}^{\infty}A_n z^n,\;\;\;A_n:=\sum_{k=0}^{n}4^{n-k}{2k \choose k}^2{2n-2k\choose n-k}.\]
Its {\em Hadamard square} is the function
\[\Phi(z):=\sum_{n=0}A_n^2 z^n\]
which satisfied the fourth order equation
\[ \sL \Phi =0\]
where $\sL$ is a rather complicated operator (Nr. 115 in the list \cite{AESZ}).
\[\]
The local system $\V$ of solutions of $\sL$ is the multiplicative convolution of the
local system $\L$ with itself and defines a VHS of type $(1,1,1,1)$.
It has only three points of maximal unipotent monodromy, $A=0,B=0,C=3$.
One has $h^1(\V)=3.3-8=1$, and hence $a=1, b=0 (!)$, and $h^{2,2}=1$.
Furthermore, the Griffiths-Yukawa coupling vanishes at a single point
($3(3-2)-2=1$).

\vskip20pt

{\small \sc Pedro Luis del Angel, CIMAT, guanjuato, Mexico}\\
{\small\em E-mail address:}  {\tt luis@cimat.mx}

\vskip 20pt
{\small \sc Stefan M\"uller-Stach, Duco van Straten and Kang Zuo},\\
{\small \sc Institut f\"ur Mathematik, Fachbereich 08, Johannes Gutenberg-Universit\"at, Mainz, Deutschland}\\ 
{\small \em E-mail address:}  {\tt \{stefan, straten, zuok\}@mathematik.uni-mainz.de}

\end{document}